\documentclass{article}
\usepackage{amsmath, amsfonts, amsthm, amssymb}
\usepackage{graphicx}
\usepackage{float}
\usepackage{verbatim}

\numberwithin{equation}{section}
\newtheorem{thm}{Theorem}[section]
\newtheorem{lma}[thm]{Lemma}
\newtheorem{cor}[thm]{Corollary}
\newtheorem{defn}[thm]{Definition}

\newtheorem{prop}[thm]{Proposition}

\newtheorem{ques}[thm]{Question}

\renewcommand{\geq}{\geqslant}
\renewcommand{\leq}{\leqslant}
\renewcommand{\H}{\text{H}}
\renewcommand{\P}{\text{P}}

\title{The Hausdorff dimension of graphs of prevalent continuous functions}

\author{J. M. Fraser and J. T. Hyde\\ \\
\emph{Mathematical Institute, University of St. Andrews, North Haugh,}\\ \emph{St. Andrews, Fife, KY16 9SS, Scotland}}

\begin{document}
\maketitle

\begin{abstract}
We prove that the Hausdorff dimension of the graph of a prevalent continuous function is 2.  We also indicate how our results can be extended to the space of continuous functions on $[0,1]^d$ for $d \in \mathbb{N}$ and use this to obtain results on the `horizon problem' for fractal surfaces.  We begin with a survey of previous results on the dimension of a generic continuous function.
\end{abstract}

\section{Introduction}

We investigate the Hausdorff dimension of the graph of a prevalent continuous function.  For $d \in \mathbb{N}$ let
\[
C[0,1]^d = \{f:[0,1]^d \to \mathbb{R}  \mid  f \text{ is continuous} \}.
\]
This is a Banach space when equipped with the infinity norm, $\| \cdot \|_\infty$.  We define the graph of a function, $f \in C[0,1]^d$, to be the set
\[
G_f = \Big\{ (x,f(x)) \mid x \in [0,1]^d \Big\} \subset \mathbb{R}^{d+1}.
\]

\subsection{Dimensions of generic continuous functions}

Over the past 25 years several papers have investigated the question:
\begin{equation} \label{genericdim}
\text{What is the `dimension' of the graph of a `generic' continuous function?}
\end{equation}
Here `dimension' could mean any of the following dimensions used to study fractal sets:
\begin{itemize}
\item[(1)] Hausdorff dimension, denoted by $\dim_\H$;
\item[(2)] lower box dimension, denoted by $\underline{\dim}_\text{B}$;
\item[(3)] upper box dimension, denoted by $\overline{\dim}_\text{B}$;
\item[(4)] lower modified box dimension, denoted by $\underline{\dim}_\text{MB}$;
\item[(5)] upper modified box dimension, denoted by $\overline{\dim}_\text{MB}$, or equivalently, packing dimension, denoted by $\dim_\P$.
\end{itemize}
For definitions and basic properties of these dimensions see \cite{falconer}.  In particular, note the following well-known proposition.

\begin{prop} \label{relationships}
For a bounded set $F \subset \mathbb{R}^d$ we have the following relationships between the dimensions discussed above:

\[
\begin{array}{ccccccccccc}
 & & & &                                                            &&                  \dim_\text{\emph{P}} F  \quad  =   \quad   \overline{\dim}_\text{\emph{MB}} F                         & & & &   \\
 & & &          &&                       \rotatebox[origin=c]{45}{$\leq$}          & &              \rotatebox[origin=c]{315}{$\leq$} & & &  \\
0 &   \leq & \dim_\text{\emph{H}} F & \leq & \underline{\dim}_\text{\emph{MB}} F                                            & &        &&                  \overline{\dim}_\text{\emph{B}} F  &     \leq & d\\
 & & &                 &&                \rotatebox[origin=c]{315}{$\leq$}              &&           \rotatebox[origin=c]{45}{$\leq$} & & &  \\
 & & & &                           &&                                          \underline{\dim}_\text{\emph{B}} F                         & & & & 
\end{array}
\]
\end{prop}

Also, there are different ways of interpreting the word `generic' in question (\ref{genericdim}).  We will focus on the following possibilities:
\begin{itemize}
\item[(1)] prevalent, i.e. `generic' from a measure theoretical point of view;
\item[(2)] typical,  i.e. `generic' from a topological point of view.
\end{itemize}

In this paper we will complete the study of question (\ref{genericdim}) in the above contexts and, in particular, show that a \emph{prevalent} continuous function has a graph with Hausdorff dimension 2.  We will also consider dimensions of graphs of prevalent functions in $C[0,1]^d$ and the `horizon problem' for prevalent surfaces.

\subsection{Prevalence}

`Prevalence' provides one way of describing the \emph{generic} behavior of a class of mathematical
objects.  In finite dimensional vector spaces Lebesgue measure provides a natural tool for deciding if a property is `generic'.  Namely, if the set of elements which do not have some property is a Lebesgue null set then it is said that this property is `generic' from a measure theoretical point of view.  However, when the space in question is infinite dimensional this approach breaks down because there is no useful analogue to Lebesgue measure in the infinite dimensional setting.  The theory of prevalence has been developed to solve this problem.  It was first introduced in the general setting of abelian Polish groups by Christensen in the 1970s \cite{christ, christ2} and later rediscovered by Hunt, Sauer and Yorke in 1992 \cite{prevalence1}.  Also, see the excellent survey paper \cite{prevalence}.
\\ \\
Since the space we are interested in, namely $(C[0,1], \| \cdot \|_\infty)$, is infinite dimensional and we wish to say something about the behavior of a \emph{generic} function it is natural to appeal to the theory of prevalence.  We will now give a brief reminder of the definitions we will need.

\begin{defn}
A \emph{completely metrizable topological vector space} is a vector space, $X$, for which there exists a metric, $d$, on $X$ such that $(X,d)$ is complete and the vector space operations are continuous with respect to the topology induced by $d$.
\end{defn}

Note that $(C[0,1], \| \cdot \|_\infty)$ is a completely metrizable topological vector space with the topology induced by the norm.

\begin{defn} \label{prevalentdef}
Let $X$ be a completely metrizable topological vector space.  A Borel set $F \subseteq X$ is \emph{prevalent} if there exists a Borel measure $\mu$ on $X$ and a compact set $K \subseteq X$ such that $0<\mu(K) < \infty$ and
\[
\mu\big(X \setminus (F+x)\big) = 0
\]
for all $x \in X$.
\\ \\
A non-Borel set $F \subseteq X$ is \emph{prevalent} if it contains a prevalent Borel set and the complement of a prevalent set is called a \emph{shy} set.
\end{defn}

Shyness is a reasonable generalisation of Lebesgue measure zero to the infinite dimensional setting.  It enjoys many of the natural properties which one would expect from such a generalisation, for example, it is preserved under taking countable unions, and, in particular, in $\mathbb{R}^d$ being shy is equivalent to having Lebesgue measure zero.  For more details see \cite{prevalence}.

\subsection{Baire category}

Prevalence is a measure theoretic approach to describing \emph{generic} behavior.  One can also consider \emph{generic} behaviour from a topological point of view using ideas from Baire category.
\\ \\
Let $X$ be a complete metric space.  A set $M$ is called \emph{meagre} if it can be written as a countable union of nowhere dense sets.  A property is called \emph{typical} if the set of points which \emph{do not} have the property is meagre.  For a more detailed account of Baire category the reader is referred to \cite{oxtoby}.

\subsection{History}

The study of different dimensional aspects of generic continuous functions has attracted much attention in the literature, see, for example, \cite{zoltan, genericmf, genericprop,  lowerprevalent, FPconj}.  In particular, over the past 25 years there has been considerable interest in answering question (\ref{genericdim}). The problem has been considered from a topological point of view, i.e., using Baire category, in \cite{humkepacking, bairefunctions, graphsums} and from a measure theoretical point of view, i.e., using prevalence, in \cite{me_horizon, lowerprevalent, mcclure, shaw}.  In fact, the question has been completely answered in the Baire category case.
\begin{thm} \label{typicalresult}
A typical function $f \in C[0,1]$ satisfies:
\[
\dim_\text{\emph{H}} G_f =  \underline{\dim}_\text{\emph{MB}} G_f = \underline{\dim}_\text{\emph{B}} G_f = \, 1 \, \, <  \, \, 2 \, = \dim_\text{\emph{P}} G_f = \overline{\dim}_\text{\emph{MB}} G_f = \overline{\dim}_\text{\emph{B}} G_f.
\]
\end{thm}

\begin{proof}
In light of Proposition \ref{relationships} it suffices to show that:
\begin{itemize}
\item[(1)] The graph of a \emph{typical} continuous function has lower box dimension 1;
\item[(2)] The graph of a \emph{typical} continuous function has packing dimension 2.
\end{itemize}

Statement (1) was proved in \cite{bairefunctions} and statement (2) was proved in \cite{humkepacking}.
\end{proof}

In the prevalence case the question has been partially answered.  It was shown in \cite{mcclure} that the packing dimension, and hence the upper box dimension, of a prevalent continuous function is 2. 
\\ \\
More recently, it has been shown in \cite{me_horizon, lowerprevalent, shaw} that the lower box dimension of the graph of a prevalent continuous function is also 2.  We remark that this result was probably first obtained in \cite{shaw}.
\\ \\
In fact, in \cite{lowerprevalent} a more general result was proved. Namely, let X be a Banach space and let $\Delta: X \to \mathbb{R}$ be a Borel measurable function such that for all $x, y \in X$ and Lebesgue almost all $t \in \mathbb{R}$ we have
\begin{equation} \label{olsencondition}
\Delta(x - ty) \geq \Delta (y).
\end{equation}
Then a prevalent element $x \in X$ satisfies:
\[
\Delta(x) = \sup_{y \in X} \Delta(y).
\]
This result was then used to show that, among other things, the lower box dimension of the graph of a prevalent function in $C[0,1]$ is as big as possible, namely 2.  Given that we want to show that the Hausdorff dimension of the graph of a prevalent function is also as big as possible it is natural to examine whether or not the function $\Delta_\H : C[0,1] \to \mathbb{R}$ defined by $\Delta_\H(f) = \dim_\H G_f$ satisfies condition (\ref{olsencondition}).  However, we have been unable to prove or disprove this and so we pose the following question:

\begin{ques}
Is it true that for all $f, g \in C[0,1]$ and for Lebesgue almost all $t \in \mathbb{R}$ we have
\[
\dim_\text{\emph{H}} G_{f - t g} \geq  \dim_\text{\emph{H}} G_{g} ?
\]
\end{ques}

Note that our methods do not rely on this approach.

\newpage

\section{Results}

\subsection{Prevalent Hausdorff dimension}

Our main result is the following.

\begin{thm} \label{main}
The set
\[
\{ f \in C[0,1] \mid \dim_\text{\emph{H}} G_f = 2 \}
\]
is a prevalent subset of $C[0,1]$.
\end{thm}

The proof of this result is deferred to Section \ref{mainproofs}.  The following corollary gives a complete answer to question (\ref{genericdim}) for prevalence.  

\begin{cor}
A prevalent function $f \in C[0,1]$ satisfies:
\[
\dim_\text{\emph{H}} G_f =  \underline{\dim}_\text{\emph{MB}} G_f = \underline{\dim}_\text{\emph{B}} G_f =  \dim_\text{\emph{P}} G_f = \overline{\dim}_\text{\emph{MB}} G_f = \overline{\dim}_\text{\emph{B}} G_f = 2.
\]
\end{cor}

\begin{proof}
This follows immediately from Theorem \ref{main} and Proposition \ref{relationships}.
\end{proof}

This result should be compared with Theorem \ref{typicalresult}.  In particular, the Hausdorff dimension of the graph of a \emph{typical} continuous function and that of a \emph{prevalent} continuous function are as different as possible.  In fact, measure theoretic and topological approaches often give contrasting answers to questions involving generic behaviour.  For example, the $L^q$-dimensions of a \emph{generic measure} were considered by Olsen from a topological point of view in \cite{typlq1, typlq2}, and from a measure theoretic point of view in \cite{prevlq}, and starkly different results were obtained.
\\ \\
Although a prevalent function, $f \in C[0,1]$, satisfies $\dim_\H G_f = 2$, it follows immediately from Fubini's Theorem that $\mathcal{H}^{2}(G_f) = 0$ for all $f \in C[0,1]$, where $\mathcal{H}^{2}$ denotes 2-dimensional Hausdorff measure.  In light of this it may be interesting to investigate the Hausdorff dimension of the graph of a prevalent continuous function using different \emph{gauge functions}.  For example, it may be true that the graph of a prevalent continuous function has positive and finite $\mathcal{H}^h$-measure for a gauge function something like
\[
h(t) = t^{2} \log \log (1/t)
\]
thus indicating that the graph of a prevalent continuous function has dimension `logarithmically smaller' than 2.  For more details on this finer approach to Hausdorff dimension see \cite{falconer, rogers}.
\\ \\
We also obtain the following higher dimensional analogue of Theorem \ref{main}.

\begin{thm} \label{higherdim}
Let $d \in \mathbb{N}$.  The set
\[
\{ f \in C[0,1]^d \mid \dim_\text{\emph{H}} G_f = d+1 \}
\]
is a prevalent subset of $C[0,1]^d$.
\end{thm}

The proof of this is very similar to the proof of Theorem \ref{main} and, therefore, we only give a sketch proof in Section \ref{smallproofs}.

\subsection{The horizon problem}

The `horizon problem'  is the study of the relationship between the dimension of the graph of a fractal surface, $f \in C[0,1]^2$, and the dimension of the graph of its horizon.

\begin{defn}
Let $f \in C[0,1]^2$.  The horizon function, $H(f) \in C[0,1]$, of $f$ is defined by
\[
H(f)(x) = \sup_{y \in [0,1]} f(x,y).
\]
\end{defn}
A `rule of thumb' is that the dimension of the horizon should be one less than the dimension of the surface.  In this case we will say that the surface satisfies the `horizon property'.  Note that the horizon property clearly does not hold in general.
\\ \\
The horizon problem was considered from a generic point of view in \cite{me_horizon}.  In particular, it was shown that a prevalent function in $C[0,1]^2$ satisfies the horizon property for box dimension.  Here we obtain the following result.

\begin{thm} \label{hausdorffhorizon}
The set
\[
\{f \in C[0,1]^2 \mid \dim_\text{\emph{H}} G_f = 3  \text{ and } \dim_\text{\emph{H}} G_{H(f)} = 2\}
\]
is a prevalent subset of $C[0,1]^2$.  In particular, a prevalent function satisfies the horizon property for Hausdorff dimension (and packing dimension).
\end{thm}

The proof of this is deferred to Section \ref{horproof}.
\\ \\
Since a prevalent surface has Hausdorff dimension as big as possible, namely 3, this does not give us any information about the horizon dimensions of surfaces with Hausdorff dimension strictly less than 3.  In \cite{me_horizon} this problem was overcome by considering subspaces of $C[0,1]^2$ indexed by $\alpha\in [2,3]$ defined by
\[
C_\alpha[0,1]^2 = \{f \in C[0,1]^2 \mid \overline{\dim}_\text{B} G_{f} \leq \alpha\}.
\]
It was shown that the subset of $C_\alpha[0,1]^2$ consisting of functions which satisfy the horizon property (for box dimension) is \emph{not} prevalent.  In our case this generalisation is not possible because the set
\[
\{f \in C[0,1]^2 \mid \dim_\H G_{f} \leq \alpha\}
\]
is not a vector space for $\alpha<3$ since it is not closed under addition.  To see this, note that it was shown in \cite{graphsums} that the graph of a \emph{typical} function, $f \in C[0,1]$, has Hausdorff dimension 1 and a trivial modification of the arguments used gives that that the graph of a \emph{typical} function, $f \in C[0,1]^2$, has Hausdorff dimension 2 (also as low as possible).  From this it follows that every function $f \in C[0,1]^2$, can be written as the sum of two functions whose graphs have Hausdorff dimension 2.  Hence, for all $\alpha<3$, we can find two functions whose graphs have Hausdorff dimension 2 but such that their sum is not in the set given above.  For more details on this see \cite{me_horizon} and the references therein.

\section{Preliminary results and notation}

In Sections 3.1-3.4 we will introduce various concepts and notation that we will use in Section 4 when proving Theorems \ref{main}, \ref{higherdim} and \ref{hausdorffhorizon}.

\subsection{Potential theoretic approach} \label{potential}

Potential theoretic methods provide a powerful tool for finding lower bounds for Hausdorff dimension.  Let $s \geq 0$ and let $\mu$ be a probability measure on $\mathbb{R}^d$.  The $s$-energy of $\mu$ is defined by
\[
I_s(\mu) =  \iint \frac{d\mu(x) \, d\mu(y)}{\lvert x-y \rvert^s}.
\]
The following theorem relates the Hausdorff dimension of a set, $F$, to the $s$-energy of probability measures supported on $F$.
\begin{thm} \label{potential}
Let $F \subset \mathbb{R}^d$ be a Borel set.  If there exists a Borel probability measure $\mu$ on $F$ with $I_s(\mu)<\infty$, then $\mathcal{H}^s(F) = \infty$ and therefore $\dim_\text{\emph{H}} F \geq s$.
\end{thm}

For a proof of this result see \cite{falconer}.

\subsection{The fat Cantor set $F$ and the measure $\nu$} \label{prelim}

In this section we will construct a `Cantor like' subset of $[0,1]$ which we will call $F$.  In Section 3.3 we will use $F$ to construct a compact set of continuous functions whose graphs `almost surely' have Hausdorff dimension 2.  We will write $\mathcal{L}^1$ to denote 1-dimensional Lebesgue measure.  Let $E_0 = [0,1]$ and let $(E_k)_{k=1}^\infty$ be a decreasing sequence of sets with $\{0,1\} \subset E_k$ for all $k$ and $(c_k)_{k=1}^\infty$ be a decreasing sequence of positive real numbers converging to 0 such that: 
\begin{itemize}
\item[(1)]  $[0,1] = E_0 \supset E_1  \supset E_2 \dots$;
\item[(2)]  For each $k \geq 1$ we have $E_k = \bigcup_{I \in \mathcal{I}_k} I$ where $\mathcal{I}_k = \{I_{k,1}, I_{k,2}, \dots, I_{k,3^{3^k}} \}$ is a collection of $3^{3^k}$ equally spaced disjoint closed intervals, each of length $c_k$;
\item[(3)]  For $k\geq 1$ we have $1 > \mathcal{L}^1(E_k) = 3^{3^k}c_k$ and, as $k \to \infty$, $3^{3^k}c_k \searrow \lambda$ for some $\lambda>0$.
\end{itemize}
Finally, let
\[
F = \bigcap_{k} E_k.
\]


It follows from (3) that $\mathcal{L}^1(F)=\lambda >0$ and therefore, setting $\nu = \mathcal{L}^1 \vert_F$, we have that
\begin{equation} \label{fatenergy}
I_{1-\varepsilon}(\nu)<\infty
\end{equation}
for all $\varepsilon \in (0,1]$.
\\ \\
%
%
Let $\mathcal{I} = \bigcup_{k\geq 1} \mathcal{I}_k$ denote the set of all construction intervals for $F$ and code the points in $F$ in the usual way, i.e., for $x \in F$, write
\[
x = (i_1(x), i_2(x), i_3(x), \dots)
\]
where $i_k(x) \in \{1,\dots, 3^{3^k} \}$ is such that $x \in I_{k,i_k(x)}$ and write $i_0(x)=1$ for all $x \in F$.
Also, for each pair $x,y \in F$ with $x \neq y$, let
\[
n(x,y) = \max \{k \mid i_k(x) = i_k(y) \},
\]
i.e., $n(x,y)$ is the last integer $k \geq 0$ such that $x$ and $y$ are in the same interval of $E_k$.

\begin{lma} \label{boundnxy}
Let $x,y \in F$ be such that $x \neq y$ and $n(x,y)\geq 1$.  Then
\[
n(x,y) \leq \log \, \log \, \lvert x-y \rvert^{-1}.
\]
\end{lma}
\begin{proof}
Let $x,y \in F$ be such that $x \neq y$ and $n(x,y)\geq 1$.  Since $x,y \in I_{n(x,y),l}$ for some $l$ and $1 \geq 3^{3^{n(x,y)}}c_{n(x,y)}$, it follows that
\[
\lvert x-y \rvert \leq c_{n(x,y)} \leq 3^{-3^{n(x,y)}}
\]
and taking logarithms gives the result.
\end{proof}

It follows from Lemma \ref{boundnxy} that for all $\varepsilon>0$ we can choose a positive constant $C_\varepsilon$ such that for all $x,y \in F$ we have
\begin{equation} \label{nxybound}
2^{n(x,y)} \leq C_\varepsilon \, \lvert x-y \rvert^{-\tfrac{\varepsilon}{2}}.
\end{equation}

\subsection{The probability space $(\Omega, \mathcal{F}, \mathbb{P})$, the compact set $K$ and the measure $\mu$}

In this section we will construct a Borel probability measure $\mu$ supported by a compact set $K \subset C[0,1]$ which we will use to \emph{witness} the prevalence of functions whose graph has Hausdorff dimension 2.  The basic idea is to construct the compact set $K$ such that, given $x,y \in [0,1]$, it is `likely' (with respect to $\mu$) that the images of $x$ and $y$ under a function in $K$ are relatively far apart compared to $\lvert x-y \rvert$.
\\ \\
Let
\[
\Omega = \Big\{ \omega \mid \mathcal{I} \to \{0,1\} \Big\} = \{0,1\}^{\mathcal{I}}
\]
be the set of all labellings of the construction intervals by 0s and 1s and equip it with the product topology, $\mathcal{T}$.   Also, let $\mathcal{F} = \sigma(\mathcal{T})$ be the Borel $\sigma$-algebra generated by $\mathcal{T}$.  We construct a probability measure on $(\Omega, \mathcal{F})$ in the natural way.  For $I \in \mathcal{I}$ and $i \in \{0,1\}$ define the 1-cylinder, $\Omega_{I,i}$, by
\[
\Omega_{I,i} = \Big\{ \omega \in \Omega \mid \omega(I) = i \Big\}
\]
and define a $k$-cylinder to be any non-empty intersection of $k$ distinct 1-cylinders.  Finally, let $\mathbb{P}_0$ be the natural mass distribution on the set of $k$-cylinders which assigns each $k$-cylinder mass $2^{-k}$.  By Kolmogorov's Consistency Theorem $\mathbb{P}_0$ extends uniquely to a probability measure $\mathbb{P}$ on $(\Omega, \mathcal{F})$.
\\ \\
To each $\omega \in \Omega$ we will associate a function $\phi_\omega \in C[0,1]$.  This function is defined as follows.  If $x \in F$ let
\[
\phi_\omega (x) = \sum_{k=1}^{\infty} 2^{-k} \omega \, \big(I_{k,i_k(x)}\big)
\]
and to extend $\phi_\omega$ to a function on $[0,1]$ we interpolate linearly on the end points of the complimentary intervals in the construction of $F$.  Let $\Phi: \Omega \to C[0,1]$ be defined by $\Phi(\omega) = \phi_\omega$ and observe that it is a continuous map.  Finally, let
\[
K= \Phi(\Omega)
\]
and
\[
\mu = \mathbb{P} \circ \Phi^{-1}.
\]

\begin{lma} \label{compactK}
The set $K$ is a compact subset of $(C[0,1], \| \cdot \|_\infty)$ and $\mu$ is a Borel measure on $C[0,1]$ which is supported by $K$.
\end{lma}

\begin{proof}
Observe that $\mathbb{P}$ is a Borel measure on $(\Omega, \mathcal{F})$ and also that $\Omega$ is compact by Tychonoff's Theorem and, therefore, the result follows immediately by the continuity of $\Phi$.
%
\end{proof}

\subsection{The measures $\nu_{\omega,f}$}

For $f \in C[0,1]$ and $\omega \in \Omega$ we will define a Borel measure $\nu_{\omega,f}$ on $\mathbb{R}^2$, with support $G_{\phi_\omega+f}$, by `lifting' the measure $\nu$ supported on the fat Cantor set $F$.
\\ \\
Let $f \in C[0,1]$ and $\omega \in \Omega$ and define the map $F_{\omega,f}:[0,1] \to \mathbb{R}^2$ by $F_{\omega,f}(x) = \big(x,\phi_\omega(x)+f(x)\big)$ and observe that it is continuous.  We may therefore define a Borel probability measure, $\nu_{\omega,f}$, on $\mathbb{R}^2$, with support $G_{\phi_\omega+f}$, by
\[
\nu_{\omega,f} = \nu \circ F_{\omega,f}^{-1}.
\]
In Section 4 we will show that the Hausdorff dimension of $G_{\phi_\omega+f}$ is $\mathbb{P}$-almost surely (and hence $\mu$-almost surely) 2 by considering the expectation of the energy of the measures $\nu_{\omega,f}$.

\section{Proofs} 

\subsection{Proof of Theorem \ref{main}}  \label{mainproofs}

In this section we will prove Theorem \ref{main}.  The proof of this will be straightforward once we have proved Lemma \ref{key}.  We will use the potential theoretic methods introduced in Section \ref{potential}.  Before proving Lemma \ref{key} we will provide the two key integral estimates (Lemmas \ref{realintegral}-\ref{integralest}).

\begin{lma} \label{realintegral}
For all $p, q \in (0,1], \, r \in \mathbb{R}$ and $\varepsilon \in (0,1/4)$ we have
\[
\int_{0}^{p} \, \int_{0}^{p}  \frac{d \alpha \,  d \beta }{ \Big( q^2 +  \lvert \alpha-\beta +r \rvert^2\Big)^{1-\varepsilon}} \leq   \frac{6p}{q^{1-\varepsilon}}.
\]
\end{lma}

\begin{proof}
We have
\begin{eqnarray*}
\int_{0}^{p} \, \int_{0}^{p}  \frac{d\alpha \,  d\beta }{ \Big(q^2 +  \lvert \alpha-\beta+r \rvert^2\Big)^{1-\varepsilon}} &=&  \int_{0}^{1} \, \int_{0}^{1}  \frac{p^2 d\alpha \,  d\beta }{ \Big(q^2 +  p^2\lvert \alpha-\beta +p^{-1} r  \rvert^2\Big)^{1-\varepsilon}} \\ \\
&\leq& p^2 \int_{0}^{1}  \frac{d\alpha }{ \Big(q^2 +  p^2\lvert \alpha - \tfrac{1}{2} \rvert^2\Big)^{1-\varepsilon}} \\ \\
&\leq& 2   p^2 \int_{0}^{1/2}  \frac{d\alpha }{ \Big( \max \big\{q^2, p^2 \alpha^2 \big\} \Big)^{1-\varepsilon}} \\ \\
&=&2  p^2 \int_{0}^{q/p}  \frac{d\alpha }{ q^{2-2\varepsilon}} \,  + \,   2  p^2  \int_{q/p}^{1/2}  \frac{d\alpha }{ p^{2-2\varepsilon} \alpha^{2-2\varepsilon} }\\ \\
&\leq&  \frac{6p}{q^{1-\varepsilon}}
\end{eqnarray*}
where the last inequality is obtained by calculating the previous integrals explicitly.
\end{proof}

In the remainder of this section we will write $\mathbb{E}$ to denote expectation with respect to $\mathbb{P}$, that is, if $X: \Omega \to \mathbb{R}$ is an integrable random variable on $\Omega$, then we will write
\[
\mathbb{E} \Big( X(\omega)\Big) = \int_\Omega  X(\omega) \, d \mathbb{P}(\omega).
\]

\begin{lma} \label{integralest}
For all $x,y \in F$, $f \in C[0,1]$ and $\varepsilon \in (0,1/4)$ we have
\begin{eqnarray*}
\mathbb{E} \, \bigg( \Big(\lvert x-y \rvert^2 + \lvert \phi_\omega(x)+f(x)-\phi_\omega(y) -f(y)\rvert^2\Big)^{\varepsilon-1} \bigg) \leq  \frac{6 \, C_\varepsilon }{\lvert x-y \rvert^{1-\tfrac{\varepsilon}{2}}}
\end{eqnarray*}
where $C_\varepsilon$ is the constant, depending only on $\varepsilon$, introduced at the end of Section \ref{prelim}.
\end{lma}

\begin{proof}
Let $x,y \in F$, $f \in C[0,1]$ and $\varepsilon \in (0,1/4)$.  In order to simplify notation let $c = f(x)-f(y) $.
\\ \\
Since $\omega(i_k(x)) = \omega(i_k(y))$ for all $k \leq n(x,y)$ we have
\begin{eqnarray*}
\phi_\omega(x)-\phi_\omega(y) &=& \sum_{k=1}^{\infty} \, 2^{-k} \, \omega\, \big( I_{k,i_k(x)} \big) - \sum_{k=1}^{\infty} \, 2^{-k} \, \omega\, \big( I_{k,i_k(y)} \big) \\ \\
&=& \sum_{k=n(x,y)+1}^{\infty} \, 2^{-k} \, \omega\, \big( I_{k,i_k(x)} \big) - \sum_{k=n(x,y)+1}^{\infty} \, 2^{-k} \, \omega\, \big( I_{k,i_k(y)} \big)\\ \\
&=& X(\omega)- Y(\omega)
\end{eqnarray*}
where $X$ and $Y$ are two independent random variables on $(\Omega, \mathcal{F}, \mathbb{P})$ with uniform distribution on the interval $[0, 2^{-n(x,y)}]$.  It follows that taking the expectation of any expression involving $X$ or $Y$ is equivalent to integrating over the interval $[0, 2^{-n(x,y)}]$ with respect to Lebesgue measure scaled by $2^{n(x,y)}$.  Therefore
\begin{eqnarray*}
\mathbb{E} \, \bigg( \Big(\lvert x-y \rvert^2 + \lvert \phi_\omega(x)-\phi_\omega(y) +c\rvert^2\Big)^{\varepsilon-1} \bigg) &=& \mathbb{E} \, \bigg( \Big(\lvert x-y \rvert^2 + \lvert X(\omega)-Y(\omega) +c\rvert^2\Big)^{\varepsilon-1} \bigg) \\ \\
&=& 4^{n(x,y)} \, \int_{0}^{2^{-n(x,y)}} \, \int_{0}^{2^{-n(x,y)}} \frac{d\alpha \,  d\beta }{ \Big(\lvert x-y \rvert^2 +  \lvert \alpha-\beta +c \rvert^2\Big)^{1-\varepsilon}} 
\end{eqnarray*}
and, by applying Lemma \ref{realintegral} with $p = 2^{-n(x,y)}$, $q=\lvert x-y \rvert$ and $r = c = f(x)-f(y)$ combined with the estimate (\ref{nxybound}), we obtain
\begin{equation*}
\mathbb{E} \, \bigg( \Big(\lvert x-y \rvert^2 + \lvert \phi_\omega(x)-\phi_\omega(y) +c\rvert^2\Big)^{\varepsilon-1} \bigg) \leq 6 \frac{ 2^{n(x,y)}}{\lvert x-y \rvert^{1-\varepsilon}}\leq  6 \frac{C_\varepsilon \lvert x-y \rvert^{-\tfrac{\varepsilon}{2}}}{\lvert x-y \rvert^{1-\varepsilon}}= \frac{6 \, C_\varepsilon }{\lvert x-y \rvert^{1-\tfrac{\varepsilon}{2}}}
\end{equation*}
completing the proof.  \end{proof}

\begin{lma} \label{key}
Let $f \in C[0,1]$.  For $\mu$-almost all $\phi \in K$ we have
\[
\dim_\text{\emph{H}} G_{\phi+f} = 2.
\]
\end{lma}

\begin{proof}
Fix $f \in C[0,1]$ and let $\varepsilon \in (0,1/4)$.  It suffices to show that for $\mathbb{P}$-almost all $\omega \in \Omega$ the measure $\nu_{\omega,f}$ has finite $(2-2\varepsilon)$-energy, i.e.,
\[
I_{2-2\varepsilon} (\nu_{\omega,f}) < \infty.
\]
Let $\omega \in \Omega$ and note that we have the following expression for $I_{2-2\varepsilon} (\nu_{\omega,f})$.
\begin{eqnarray}
I_{2-2\varepsilon} (\nu_{\omega,f}) &=& \int_ {\textbf{x} \in G_{\phi_\omega+f}} \int_ {\textbf{y} \in G_{\phi_\omega+f}}  \frac{d\nu_{\omega,f}(\textbf{x})  \, d\nu_{\omega,f}(\textbf{y})}{\lvert \textbf{x}-\textbf{y} \rvert^{2-2\varepsilon}} \nonumber \\ \nonumber \\
&=& \int_ {\textbf{x} \in G_{\phi_\omega+f}} \int_ {\textbf{y} \in G_{\phi_\omega+f}}   \frac{d\big(\nu \circ F_{\omega,f}^{-1}\big)(\textbf{x})  \, \, d\big(\nu \circ F_{\omega,f}^{-1}\big)(\textbf{y}) }{ \bigg(\Big\lvert F_{\omega,f}^{-1}(\textbf{x})-F_{\omega,f}^{-1}(\textbf{y}) \Big\rvert^2 + \Big\lvert (\phi_\omega+f)\big(F_{\omega,f}^{-1}(\textbf{x})\big)-(\phi_\omega+f)\big(F_{\omega,f}^{-1}(\textbf{y})\big) \Big\rvert^2\bigg)^{1-\varepsilon}} \nonumber \\ \nonumber \\ 
&=& \int_ 0^1 \int_ 0^1  \frac{d\nu (x)  \, d\nu (y) }{ \Big(\lvert x-y \rvert^2 + \lvert (\phi_\omega+f)(x)-(\phi_\omega +f)(y)\rvert^2\Big)^{1-\varepsilon}} \nonumber \\ \nonumber \\ 
&=& \int_ F \int_ F  \frac{d\nu (x)  \, d\nu (y) }{ \Big(\lvert x-y \rvert^2 + \lvert \phi_\omega(x) +f(x)-\phi_\omega(y) -f(y)\rvert^2\Big)^{1-\varepsilon}}. \label{project}
\end{eqnarray}

It follows from Lemma \ref{integralest} and (\ref{fatenergy}) that
\begin{eqnarray}
\int_ F \int_ F \mathbb{E} \, \bigg( \Big(\lvert x-y \rvert^2 + \lvert \phi_\omega(x)+f(x)-\phi_\omega(y) -f(y) \rvert^2\Big)^{\varepsilon-1} \bigg)  \, d\nu (x)  \, d\nu (y)&\leq& \int_ F \int_ F  \frac{6 \, C_\varepsilon  }{\lvert x-y \rvert^{1-\tfrac{\varepsilon}{2}}} \, d\nu (x)  \, d\nu (y) \nonumber\\ \nonumber \\
&=& 6 \, C_\varepsilon  \, I_{1-\tfrac{\varepsilon}{2}}(\nu) \nonumber \\  \nonumber \\
&<& \infty. \label{finite}
\end{eqnarray}
Since $\mathbb{P}$ and $\nu$ are finite measures and the integral above is finite, we can apply Fubini's Theorem to deduce that
\begin{eqnarray*}
\mathbb{E} \, \bigg( I_{2-2\varepsilon}(\nu_{\omega,f})  \bigg) &=& \mathbb{E} \, \Bigg(  \int_ F \int_ F  \frac{d\nu (x)  \, d\nu (y) }{ \Big(\lvert x-y \rvert^2 + \lvert \phi_\omega(x)+f(x)-\phi_\omega(y)-f(y) \rvert^2\Big)^{1-\varepsilon}} \Bigg)  \qquad \text{by (\ref{project})}\\ \\
&=& \int_ F \int_ F \mathbb{E} \, \bigg(\Big(\lvert x-y \rvert^2 + \lvert \phi_\omega(x)+f(x)-\phi_\omega(y)-f(y) \rvert^2\Big)^{\varepsilon-1} \bigg)\, d\nu (x)  \, d\nu (y) \qquad \text{by Fubini}\\ \\
&<& \infty
\end{eqnarray*}
by (\ref{finite}).  It follows that for $\mathbb{P}$-almost all $\omega \in \Omega$ we have $I_{2-2\varepsilon} (\nu_{\omega,f}) < \infty$.  This, combined with Theorem \ref{potential} and the fact that $\varepsilon \in (0,1/4)$ was arbitrary, proves the result.
\end{proof}

We obtain the following corollary.

\begin{cor}
For $\mu$-almost all $\phi \in K$ we have
\[
\dim_\text{\emph{H}} G_{\phi} = 2.
\]
\end{cor}

\begin{proof}
This follows immediately by applying Lemma \ref{key} with $f \equiv 0$.
\end{proof}

The proof of Theorem \ref{main} now follows easily.  We begin with a technical lemma.
\begin{lma} \label{borelset}
The set
\[
\{ f \in C[0,1] \mid \dim_\text{\emph{H}} G_f = 2 \}
\]
is a Borel subset of $(C[0,1], \| \cdot \|_\infty)$.
\end{lma}

\begin{proof}
Let $\mathcal{K}(\mathbb{R}^d)$ denote the set consisting of all non-empty compact subsets of $\mathbb{R}^d$ and equip this space with the Hausdorff metric, $d_\mathcal{H}$.  It was shown in \cite{mattilamauldin} that the function $\Delta_{d,\H}: (\mathcal{K}(\mathbb{R}^d), d_\mathcal{H}) \to \mathbb{R}$ defined by
\[
\Delta_{d,\H} (K) = \dim_\H K
\]
is of Baire class 2, and, in particular, Borel measurable.  Define a map $\Gamma: (C[0,1],\| \cdot \|_\infty) \to (\mathcal{K}(\mathbb{R}^2),d_\mathcal{H})$ by
\[
\Gamma(f) = G_f.
\]
It is easily shown that $\Gamma$ is continuous, and therefore Borel, and hence the composition $\Delta_{2,\H} \circ \Gamma$ is Borel measurable.  It follows that
\[
(\Delta_{2,\H} \circ \Gamma)^{-1}(\{2\}) = \{ f \in C[0,1] \mid \dim_\H G_f = 2 \}
\]
is a Borel set.
\end{proof}

Theorem \ref{main} now follows immediately from Lemma \ref{borelset}, Lemma \ref{compactK} and Lemma \ref{key} since, writing $A=\{ f \in C[0,1] : \dim_\H G_f = 2 \}$, we have, for all $f \in C[0,1]$
\[
\mu \Big(C[0,1] \setminus (A+f) \Big) = \mu \Big( \{\phi \in K \mid  \dim_\H G_{ \phi - f} < 2\} \Big) = 0.
\]

\subsection{Sketch proof of Theorem \ref{higherdim}} \label{smallproofs}

Let $d \in \mathbb{N}$, $\omega \in \Omega$ and define a function $\phi_{d,\omega}: [0,1]^d \to \mathbb{R}$ by
\[
\phi_{d,\omega}(x_1, \dots, x_d) = \phi_\omega(x_1).
\]
Also, define a map $\Phi_d: \Omega \to C[0,1]^d$ by
\[
\Phi_d(\omega) = \phi_{d,\omega}
\]
and note that it is continuous.  Finally, let $K_d = \Phi_d(\Omega)$ and $\mu_d = \mathbb{P} \circ \Phi_d^{-1}$.  Using $K_d$ and $\mu_d$ in place of $K$ and $\mu$, Theorem \ref{higherdim} can now be proved in a very similar way to Theorem \ref{main}.  In particular, we obtain the following analogue of Lemma \ref{key}.

\begin{lma} \label{key2}
Let $f \in C[0,1]^d$.  For $\mu_d$-almost all $\phi \in K_d$ we have
\[
\dim_\text{\emph{H}} G_{\phi+f} = d+1.
\]
\end{lma}

\subsection{Proof of Theorem \ref{hausdorffhorizon}} \label{horproof}

Write
\[
B=\{f \in C[0,1]^2 \mid \dim_\H G_f = 3  \text{ and } \dim_\H G_{H(f)} = 2\} =B_1 \cap B_2 
\]
where
\[
B_1=\{f \in C[0,1]^2 \mid \dim_\H G_f = 3 \}
\]
and
\[
B_2=\{f \in C[0,1]^2\mid \dim_\H G_{H(f)} = 2\}.
\]
Note that $B_1$ is prevalent by Theorem \ref{higherdim} and so, since the intersection of two prevalent sets is prevalent, it suffices to show that $B_2$ is prevalent.  Let $K_2$ and $\mu_2$ be the compact set and Borel probability measure described in Section \ref{smallproofs} and let $f \in C[0,1]^2$.  Note that, for all $\phi_{2,\omega} \in K_2$ and $f \in C[0,1]^2$, since $\phi_{2,\omega}(x,y)$ is independent of $y$, we have
\begin{eqnarray}
H(f+\phi_{2,\omega})(x) = \sup_{y \in [0,1]} (f+\phi_{2,\omega})(x,y) &=& \sup_{y \in [0,1]} \Big(f(x,y)+\phi_{2,\omega}(x,y) \Big) \nonumber \\ \nonumber \\
&=&\sup_{y \in [0,1]} \Big(f(x,y)\Big)+H(\phi_{2,\omega})(x) \nonumber \\ \nonumber \\
&=&H(f)(x)+\phi_\omega(x). \label{horizonss} 
\end{eqnarray}
It follows that for all $f \in C[0,1]^2$ we have
\begin{eqnarray*}
\mu_2 \Big(C[0,1]^2 \setminus (B_2+f) \Big) &=&  \mu_2\Big( K_2 \setminus (B_2+f)  \Big) \\ \\
&=& \big(\mathbb{P} \circ \Phi_2^{-1} \big) \Big( \{ \phi \in K_2 :  \dim_\H G_{ H(\phi-f)} < 2 \} \Big) \\ \\
&=&  \mathbb{P} \, \Big( \{ \omega \in \Omega \mid  \dim_\H G_{ \phi_{\omega}-H(f)} < 2 \}\Big)  \qquad \qquad \text{by (\ref{horizonss})} \\ \\
&=& 0
\end{eqnarray*}
by Lemma \ref{key}. \hfill \qed

\begin{centering}

\textbf{Acknowledgements}

\end{centering}

We thank Kenneth Falconer and Lars Olsen for many helpful discussions and comments regarding the work presented here.

\end{document}